\RequirePackage{fix-cm}
\documentclass[smallextended]{svjour3}       
\smartqed  
\usepackage{color}
\usepackage[colorlinks=true,
            citecolor=blue,
            linkcolor=blue,
            anchorcolor=green,
            urlcolor=blue
            ]{hyperref}
\usepackage{graphicx}
\usepackage{lipsum}
\usepackage{amsfonts}
\usepackage{graphicx}
\usepackage{epstopdf}
\usepackage{algorithm}
\usepackage{algorithmic}
\usepackage{amsmath}
\usepackage{amssymb}
\usepackage{bm}
\usepackage{mathbbold}
\usepackage{mathrsfs}
\usepackage{multirow}
\usepackage{dcolumn}
\usepackage{listings}
\usepackage{subfigure}
\usepackage{enumerate}
\usepackage{marvosym}
\numberwithin{theorem}{section}
\numberwithin{corollary}{section}
\numberwithin{lemma}{section}
\numberwithin{remark}{section}
\numberwithin{definition}{section}

\newcommand{\ud}{\mathrm{d}}

\newtheorem{assum}{Assumption~}
\usepackage{amsopn}

\usepackage{mathptmx}      
\usepackage{latexsym}
\begin{document}

\title{Derivative-free global minimization for a class of multiple minima problems}

\titlerunning{Derivative-free global minimization}        

\author{Xiaopeng Luo \and Xin Xu \and Daoyi Dong}


\institute{\Letter~Xin Xu \\
\email{xu.permanent@gmail.com} \\
\at Department of Control and Systems Engineering, School of Management and Engineering, Nanjing University, Nanjing, 210008, China \\
Department of Chemistry, Princeton University, Princeton, NJ 08544, USA
\and
Xiaopeng Luo \\
\email{luo.permanent@gmail.com} \\
\at Department of Control and Systems Engineering, School of Management and Engineering, Nanjing University, Nanjing, 210008, China \\
Department of Chemistry, Princeton University, Princeton, NJ 08544, USA
\and
Daoyi Dong \\
\email{daoyidong@gmail.com} \\
\at School of Engineering and Information Technology, University of New South Wales, Canberra, 2600, Australia
}

\date{}

\maketitle

\begin{abstract}
  We prove that the finite-difference based derivative-free descent (FD-DFD) methods have a capability to find the global minima for a class of multiple minima problems. Our main result shows that, for a class of multiple minima objectives that is extended from strongly convex functions with Lipschitz-continuous gradients, the iterates of FD-DFD converge to the global minimizer $x_*$ with the linear convergence $\|x_{k+1}-x_*\|_2^2\leqslant\rho^k \|x_1-x_*\|_2^2$ for a fixed $0<\rho<1$ and any initial iteration $x_1\in\mathbb{R}^d$ when the parameters are properly selected. Since the per-iteration cost, i.e., the number of function evaluations, is fixed and almost independent of the dimension $d$, the FD-DFD algorithm has a complexity bound $\mathcal{O}(\log\frac{1}{\epsilon})$ for finding a point $x$ such that the optimality gap $\|x-x_*\|_2^2$ is less than $\epsilon>0$. Numerical experiments in various dimensions from $5$ to $500$ demonstrate the benefits of the FD-DFD method.

\end{abstract}

\section{Introduction}
\label{DFD:s1}

Derivative-free descent (DFD) methods, also known as zero-order methods \cite{DuchiJ2015A_ZeroOrderCO,ShamirO2017A_ZeroOrderConvex} in the literature or bandit optimization in the machine learning literature \cite{HazanE2014A_BanditOptimization,ShamirO2017A_ZeroOrderConvex}, do not require the availability of derivatives. They are generally applied to instances where derivatives are unavailable or unreliable \cite{ConnA2009M_DerivativeFree,PoliakB1987M_Optimization}. DFD methods do not depend directly on gradient information, but some of them are related to gradient estimates, e.g., the finite-difference based derivative-free descent (FD-DFD) methods \cite{NesterovY2017A_GF,PoliakB1987M_Optimization}.

The FD-DFD method could be regarded as a smoothed extension of the gradient method because an FD-DFD descent direction is an unbiased estimate of smoothed gradient at the current iterate \cite{NemirovskiA1983M_optimization,NesterovY2017A_GF}. It also has some characteristics of the gradient method \cite{NesterovY2017A_GF} but its ``cognitive range'' is closely related to the smoothing parameter. One may expect that the FD-DFD method has certain global search capability when the smoothing and stepsize parameters are properly selected. Therefore, we attempt to analyze the convergence behavior of the FD-DFD method in global optimization under certain conditions, regardless of whether the derivative is available.

Recently, a regularized asymptotic descent (RAD) method \cite{LuoX2020A_RAD} was proposed to find the global minima with linear convergence and logarithmic work complexity for certain class of multiple minima problems. It is inspired by an asymptotic solution of the regularized minimization problems which is extended from the Pincus asymptotic solution formula \cite{PincusM1968A_AsymptoticSolution,PincusM1970A_AsymptoticSolution}. Under a mild assumption, the RAD iterates will converge to the global minimizer without being trapped in saddle points, local minima, or even discontinuities.  In this work, we will prove that the FD-DFD method also enjoys a similar convergence behavior under the same assumption.

We will see that this global linear convergence comes from the smoothing effect on objective gradients. Smoothing techniques have been extensively studied \cite{BurkeJ2020A_Max,DuchiJ2012A_SGsmoothing,NesterovY2017A_GF}, especially for nonsmooth objectives
\cite{ChenX2012A_Smoothing,NesterovY2005A_Smooth}. One may notice that a smoothing method is different from a random noise based approach, e.g., random search method \cite{BergstraJ2012M_RS_HyperParameter} or perturbed gradient method \cite{JinC2017A_EscapeSaddlePoints,PemantleR1990A_SaddlePoints}. The reason is that a smoothed gradient contains certain global information but a random noise is independent of the objective.

Specifically, we analyze the FD-DFD method for finding the global minima
\begin{equation}\label{RAD:eq:P}
  x_*=\arg\min_{x\in\mathbb{R}^d}f(x),
\end{equation}
where the objective function $f:\mathbb{R}^d\to\mathbb{R}$ satisfies the following assumption:
\begin{assum}\label{DFD:ass:A}
The objective function $f:\mathbb{R}^d\to\mathbb{R}$ satisfies there exist $x_*\in\mathbb{R}^d$ and $0<l\leqslant L<\infty$ such that for all $x\in\mathbb{R}^d$,
\begin{equation}\label{DFD:eq:A}
  f_*+\frac{l}{2}\|x-x_*\|_2^2\leqslant f(x)
    \leqslant f_*+\frac{L}{2}\|x-x_*\|_2^2.
\end{equation}
Hence, $f$ has a unique global minimizer $x_*$ with $f_*:=f(x_*)$.
\end{assum}

\begin{figure}[!h]
\centering
\includegraphics[width=0.4\textwidth]{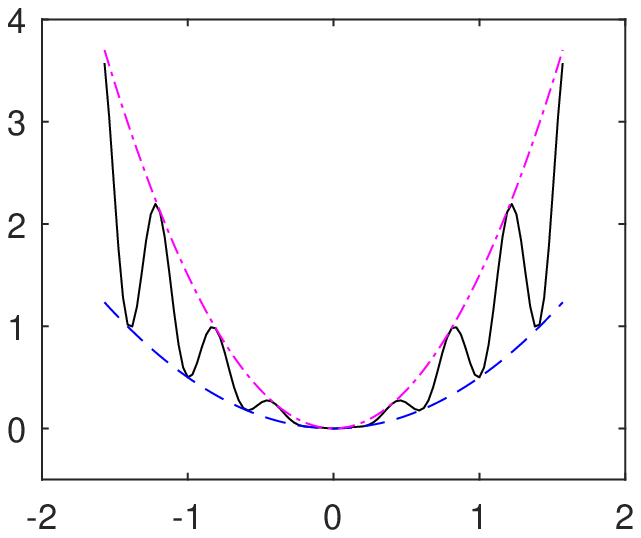}
\includegraphics[width=0.4\textwidth]{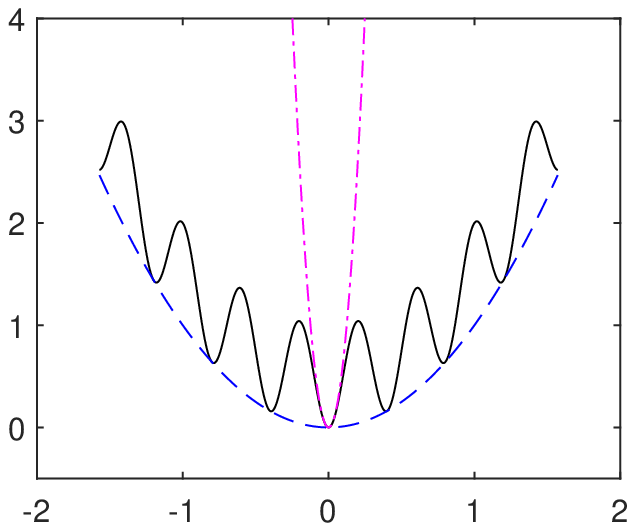}
\caption{One-dimensional examples. Left: the objective $f(x)=x^2+x^2\cos(5\pi x)/2$ (solid) with lower bound $x^2/2$ (dashed) and upper bound $3x^2/2$ (dash-dotted). Right: the objective $f(x)=x^2-\cos(5\pi x)/2+1/2$ (solid) with lower bound $x^2$ (dashed) and upper bound $65x^2$ (dash-dotted).}
\label{DFD:fig:A}
\end{figure}

As shown in Fig. \ref{DFD:fig:A}, such a class of functions can be extended from strongly convex functions with Lipschitz-continuous gradients; however, it is not ruling out the possibility of multiple minima. The lower bound $f_*+\frac{l}{2}\|x-x_*\|_2^2$ guarantees the uniqueness of the global minima while the upper bound $f_*+\frac{L}{2}\|x-x_*\|_2^2$ controls the sharpness of the minima.

Under Assumption \ref{DFD:ass:A}, the Lipschitz-continuous objective $f$ has a unique global minima $x_*$ and possibly multiple local minima. Our goal here is to find this global minima $x_*$ without being trapped in saddle points or local minima. We prove that the FD-DFD method enjoys global linear convergence, i.e., $\|x_{k+1}-x_*\|_2^2=\rho^k\|x_1-x_*\|_2^2$, for finding the global minimizer $x_*$ when the parameters are properly selected (see Theorem \ref{DFD:thm:main2}).

The remainder of the paper is organized as follows. In Sect. \ref{DFD:s2}, we establish the convergence property and complexity bound
for the FD-DFD method. In Sect. \ref{DFD:s3}, we compare the characteristics of the FD-DFD method with the RAD method. In Sect. \ref{DFD:s4}, we demonstrate the benefits of the FD-DFD method by numerical experiments in various dimensions from $5$ to $500$. And finally, we draw conclusions in Sect. \ref{DFD:s5}.

\section{Derivative-free methods}
\label{DFD:s2}

\subsection{Gradient of Gaussian smoothing}

Let random vector $\xi$ have $d$-dimensional standard normal distribution, that is, $\xi\sim\mathcal{N}(0,I_d)$. Denote by $\mathbb{E}_\xi(h(\xi))$ the expectation of corresponding random variable. For any smoothing parameter $\sigma>0$, we consider
\begin{equation*}
  g_\sigma(x,\xi)=
  \frac{\big(f(x+\sigma\xi)-f_*\big)\xi}{\sigma}
\end{equation*}
as an unbiased estimate for the gradient of the smoothed objective function
\begin{equation}\label{DFD:eq:SO}
  f_\sigma(x)=\mathbb{E}_\xi\big[f(x+\sigma\xi)\big].
\end{equation}
In practice, $f_*$ and can be replaced with corresponding estimate.

Using the substitution $\xi=\frac{\theta-x}{\sigma}$, this gradient estimate can also be written as \begin{equation*}
  g_\sigma(x,\theta)=\frac{\big(f(\theta)-f_*\big)(\theta-x)}{\sigma^2},
  ~~~\theta\sim\mathcal{N}(x,\sigma^2I_d).
\end{equation*}
See \cite{NesterovY2017A_GF,PoliakB1987M_Optimization} for other different finite-difference based schemes.

Theorem \ref{SGFD:thm:stepsmoothing} establishes $\nabla f_\sigma(x)= \frac{1}{\sigma}\mathbb{E}_\xi[f(x+\sigma\xi)\xi]= \frac{1}{\sigma^2}\mathbb{E}_\theta[f(\theta)(\theta-x)]$, thus, we obtain the unbiasedness
\begin{equation*}
  \mathbb{E}_\xi[g_\sigma(x,\xi)]=
  \mathbb{E}_\theta[g_\sigma(x,\theta)]=\nabla f_\sigma(x).
\end{equation*}

\begin{theorem}\label{SGFD:thm:stepsmoothing}
Suppose that $f_\sigma:\mathbb{R}^d\to\mathbb{R}$ is the smoothed objective with a smoothing parameter $\sigma>0$ defined by \eqref{DFD:eq:SO}, then its gradient
\begin{equation}\label{DFD:eq:Gsmoothing}
  \nabla f_\sigma(x)=\frac{1}{\sigma}
  \mathbb{E}_\xi[f(x+\sigma\xi)\xi]
  =\frac{1}{\sigma^2}\mathbb{E}_\theta[f(\theta)(\theta-x)],
\end{equation}
where $\xi\sim\mathcal{N}(0,I_d)$ and $\theta\sim\mathcal{N}(x,\sigma^2I_d)$.
\end{theorem}
\begin{proof}
Using the substitution $\xi=\frac{\theta-x}{\sigma}$, we obtain
\begin{align*}
  f_\sigma(x)=\mathbb{E}_\xi\big[f(x+\sigma\xi)\big]
  =\frac{1}{(\sqrt{2\pi}\sigma)^d}\int_{\mathbb{R}^d}
  f(\theta)e^{-\frac{\|\theta-x\|_2^2}{2\sigma^2}}\ud\theta,
\end{align*}
and therefore,
\begin{align*}
  \nabla f_\sigma(x)\!=&\!\frac{1}{(\sqrt{2\pi}\sigma)^d}
  \!\!\int_{\mathbb{R}^d}\!f(\theta)\nabla_x
  e^{-\frac{\|\theta-x\|_2^2}{2\sigma^2}}\!\ud\theta \\
  \!=&\!\frac{1}{(\sqrt{2\pi}\sigma)^d}
  \!\!\int_{\mathbb{R}^d}\!f(\theta)\!
  \left(\!\frac{\theta\!-\!x}{\sigma^2}\!\right)
  e^{-\frac{\|\theta-x\|_2^2}{2\sigma^2}}\!\ud\theta.
\end{align*}
That is, $\nabla f_\sigma(x)=\frac{1}{\sigma^2} \mathbb{E}_\theta[f(\theta)(\theta-x)]$. Using the substitution $\theta=x+\sigma\xi$, we further obtain
\begin{equation*}
  \nabla f_\sigma(x)=\frac{1}{\sigma}\frac{1}{(\sqrt{2\pi})^d}
  \int_{\mathbb{R}^d}f(x+\sigma\xi)\xi
  e^{-\frac{\|\xi\|_2^2}{2}}\ud\xi
  =\frac{1}{\sigma}\mathbb{E}_\xi[f(x+\sigma\xi)\xi],
\end{equation*}
and the proof is complete. \qed
\end{proof}

We will see that the smoothed gradient $\nabla f_\sigma$ plays a key role in the subsequent analysis. At the same time, we also notice that the Gaussian smoothing does not make any significant changes to the bounds of the objective function $f$. Specifically, under Assumption \ref{DFD:ass:A}, we have,
\begin{equation*}
  f_*+\frac{l}{2}\|x-x_*\|_2^2\leqslant f(x)\leqslant
  f_*+\frac{L}{2}\|x-x_*\|_2^2,
\end{equation*}
together with
\begin{align*}
  &\frac{1}{(\sqrt{2\pi}\sigma)^d}\int_{\mathbb{R}^d}
  \|\theta-x_*\|_2^2e^{-\frac{\|\theta-x\|_2^2}{2\sigma^2}}\ud\theta \\
  =&\frac{1}{(\sqrt{2\pi}\sigma)^d}\sum_{i=1}^d\int_{\mathbb{R}^d}
  \big(\theta^{(i)}-x^{(i)}+x^{(i)}-x_*^{(i)}\big)^2
  e^{-\frac{\|\theta-x\|_2^2}{2\sigma^2}}\ud\theta \\
  =&\frac{1}{(\sqrt{2\pi}\sigma)^d}\sum_{i=1}^d\int_{\mathbb{R}^d}
  \big(\theta^{(i)}-x^{(i)}\big)^2
  e^{-\frac{\|\theta-x\|_2^2}{2\sigma^2}}\ud\theta
  +\sum_{i=1}^d\big(x^{(i)}-x_*^{(i)}\big)^2 \\
  =&d\sigma^2+\|x-x_*\|_2^2,
\end{align*}
we obtain the bounds
\begin{equation*}
  f_*+\frac{l}{2}\|x-x_*\|_2^2+\frac{d\sigma^2l}{2}
  \leqslant f_\sigma(x)\leqslant
  f_*+\frac{L}{2}\|x-x_*\|_2^2+\frac{d\sigma^2L}{2}.
\end{equation*}

\subsection{Algorithms}

With an initial point $x_1$, a stepsize $\alpha>0$, an initial exploration radius $\lambda>0$, a fixed contraction factor $0<\rho<1$ and a number of function evaluations per-iteration $n\in\mathbb{N}$, the FD-DFD method is characterized by the iteration
\begin{equation}\label{DFD:eq:Ixk}
  x_{k+1}=x_k-\alpha g_k,
\end{equation}
where the gradient estimate
\begin{equation}\label{DFD:eq:gk}
  g_k=\frac{1}{n\sigma_k^2}\sum_{i=1}^n
  \bigg(f(\theta_{k,i})-\min_{1\leqslant j\leqslant n}
  f(\theta_{k,j})\bigg)(\theta_{k,i}-x_k),
\end{equation}
here, $\theta_{k,i}\sim\mathcal{N}(x_k,\sigma_k^2I_d)$ and $\sigma_k^2=\rho^{k}\lambda^{-1}$. To increase the stability of the iteration, we recommend to use the gradient estimate
\begin{equation}\label{DFD:eq:stablegk}
  \hat{g}_k=\frac{1}{n\hat{m}_k}\sum_{i=1}^n
  \bigg(f(\theta_{k,i})-\min_{1\leqslant j\leqslant n}
  f(\theta_{k,j})\bigg)(\theta_{k,i}-x_k),
\end{equation}
where
\begin{equation*}
  \hat{m}_k^2=\frac{1}{n}\sum_{i=1}^n \bigg(f(\theta_{k,i})-\min_{1\leqslant j\leqslant n}
  f(\theta_{k,j})\bigg)^2=\frac{1}{n}\sum_{i=1}^n
  \bigg(f(x_k+\sigma_k\xi_i)-\min_{1\leqslant j\leqslant n}
  f(\xi_j)\bigg)^2.
\end{equation*}
Here the random vector $\xi$ has $d$-dimensional standard normal distribution.

Note that, if $x_k\to x_*$ as $k\to\infty$, $\hat{m}_k^2$ can be viewed as an estimate for
\begin{equation*}
  m_k^2=\mathbb{E}_\xi\big(f(x_*+\sigma_k\xi)-f_*\big)^2.
\end{equation*}
Under Assumption \ref{DFD:ass:A}, i.e.,
\begin{equation*}
  \frac{l\sigma_k^2}{2}\|\xi\|_2^2\leqslant f(x_*+\sigma_k\xi)-f_*\leqslant
  \frac{L\sigma_k^2}{2}\|\xi\|_2^2,
\end{equation*}
we can obtain $m_k=\mathcal{O}(\sigma_k^2)$; and futher, $\hat{g}_k=\mathcal{O}(g_k)$. The difference between these two iterations based on $g_k$ and $\hat{g}_k$ is shown in Fig. \ref{DFD:fig:B}. We can see that when the stepsize parameter is properly selected, they have almost the same convergence behavior in the early stage while the estimate $\hat{g}_k$ is more stable than the estimate $g_k$ as the iteration progresses. Hence, we mainly consider the gradient estimate $\hat{g}_k$ in the following. Now we present the FD-DFD method as Algorithm \ref{DFD:alg:DFD}.

\begin{figure}[!h]
\centering
\includegraphics[width=0.325\textwidth]{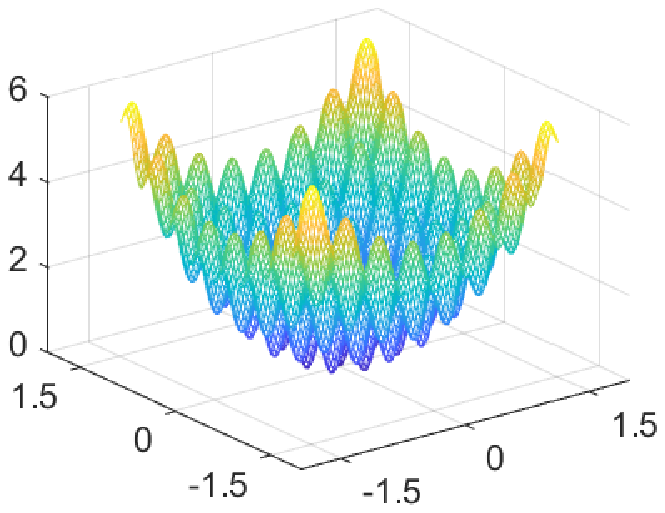}
\includegraphics[width=0.325\textwidth]{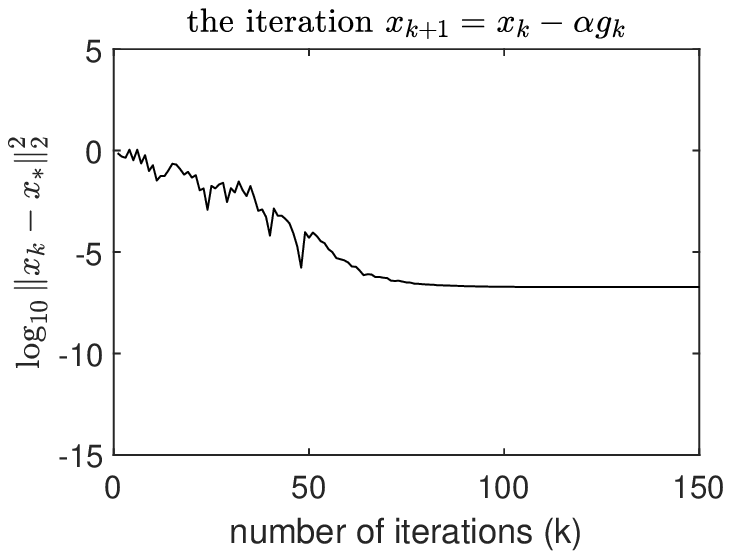}
\includegraphics[width=0.325\textwidth]{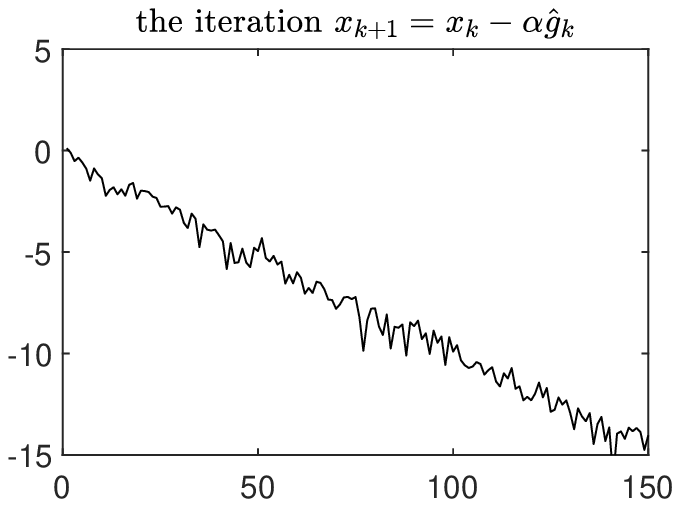}
\caption{Difference between two iterations. Left: the objective $f(x)=\|x\|_2^2-\frac{1}{2}\sum_{i=1}^2 \cos\big(5\pi x^{(i)}\big)+1,~ x\in\mathbb{R}^2$. Middle: the iteration $x_{k+1}=x_k-\alpha g_k$ with parameter settings $x_1=(1,-1)$, $\lambda=1/\sqrt{2}$, $\rho=0.9$, $n=5$ and $\alpha=0.5$. Right: the iteration $x_{k+1}=x_k-\alpha\hat{g}_k$ with parameter settings $x_1=(1,-1)$, $\lambda=1/\sqrt{2}$, $\rho=0.9$, $n=5$ and $\alpha=0.5$.}
\label{DFD:fig:B}
\end{figure}

\begin{algorithm}
\caption{FD-DFD Method}
\label{DFD:alg:DFD}
\begin{algorithmic}[1]
\STATE{Choose an initial iterate $x_1$ and preset parameters $\alpha>0$, $\lambda>0$, $\rho\in(0,1)$, $n\in\mathbb{N}$.}
\FOR{$k=1,2,\cdots$}
\STATE{Set the variance $\sigma_k=\rho^{\frac{k}{2}} \lambda^{-\frac{1}{2}}$.}
\STATE{Generate $n$ realizations $\{\theta_{k,i}\}_i^n$ of $\theta_k$ from $\mathcal{N}(x_k,\sigma_k^2I_d)$.}
\STATE{Compute a stochastic vector $\hat{g}_k$.}
\STATE{Set the new iterate as $x_{k+1}=x_k-\alpha\hat{g}_k$.}
\ENDFOR
\end{algorithmic}
\end{algorithm}

It follows from Chebyshev's inequality and $\mathbb{E}[g_k]=\nabla f_{\sigma_k}(x_k)$ that, for all $C>0$ and any $1\leqslant i\leqslant d$, with probability at least $1-\frac{1}{C^2}$, the $i$-th component of $g_k$ satisfies
\begin{equation}\label{DFD:eq:gki}
  \bigg|-g_k^{(i)}+\nabla^{(i)}f_{\sigma_k}(x_k)\bigg|\leqslant
  C\sqrt{\mathbb{V}\big[g_k^{(i)}\big]}.
\end{equation}
Furthermore, the $i$-th component of $x_{k+1}$ satisfies
\begin{equation}\label{DFD:eq:xki}
  x_{k+1}^{(i)}=x_k^{(i)}-\alpha g_k^{(i)}
  =x_k^{(i)}-\alpha\nabla^{(i)}f_{\sigma_k}(x_k).
\end{equation}

\subsection{Analyses}

To establish the expected linear convergence of the FD-DFD method, we first build upper bounds for $-\nabla^{(i)}f_{\sigma_k}(x_k)$ and $\mathbb{V}\big[g_k^{(i)}\big]$ by the following two lemmas.

\begin{lemma}\label{DFD:lem:gkbound}
Under Assumption \ref{DFD:ass:A}, suppose that $\sigma_k=\rho^{\frac{k}{2}}\lambda^{-\frac{1}{2}}$ with $0<\rho<1$ and there is an $M>0$ such that $\|x_k-x_*\|_2^2\leqslant\rho^kM$. Then for any $1\leqslant i\leqslant d$, the $i$-th component of $\nabla f_{\sigma_k}(x_k)$ satisfies the following inequality:
\begin{equation*}
  \bigg|-\nabla^{(i)}f_{\sigma_k}(x_k)
  +\frac{L+l}{2}\big(x_k^{(i)}-x_*^{(i)}\big)\bigg|
  \leqslant\frac{\rho^{\frac{k}{2}}(L-l)}{\sqrt{2\pi}}
  \left((d+2)\lambda^{-\frac{1}{2}}
  +\frac{2M}{\lambda^{-\frac{1}{2}}}\right);
\end{equation*}
further, when $\lambda^{-1}=\frac{2M}{d+2}$, this inequality can be improved as
\begin{equation*}
  \bigg|-\nabla^{(i)}f_{\sigma_k}(x_k)
  +\frac{L+l}{2}\big(x_k^{(i)}-x_*^{(i)}\big)\bigg|
  \leqslant\rho^{\frac{k}{2}}(L-l)\frac{\sqrt{(d+2)M}}{\sqrt{\pi}},
\end{equation*}
where $x^{(i)}$ be the $i$-th component of $x\in\mathbb{R}^d$.
\end{lemma}
\begin{proof}
For convenience we define $\varphi_k(x)=\prod_{i=1}^d\phi_k(x^{(i)})$, where
\begin{equation*}
  \phi_k(x^{(i)})=\frac{1}{\sqrt{2\pi}\sigma_k}
  \exp\bigg(\frac{-(x^{(i)}-x_k^{(i)})^2}{2\sigma_k^2}\bigg).
\end{equation*}
According to Theorem \ref{SGFD:thm:stepsmoothing}, for any $1\leqslant i\leqslant d$, we have
\begin{equation*}
  \nabla^{(i)}f_{\sigma_k}(x_k)
  =\frac{1}{\sigma_k^2}\int_{\mathbb{R}^d}f(x)
  \big(x^{(i)}-x_k^{(i)}\big)\varphi_k(x)\ud x.
\end{equation*}
Noting that
\begin{equation*}
  \int_{\mathbb{R}^d}\big(x^{(i)}-x_k^{(i)}\big)\varphi_k(x)\ud x=0,
\end{equation*}
it follows that
\begin{align*}
  \nabla^{(i)}f_{\sigma_k}(x_k)
  =&\frac{1}{\sigma_k^2}\int_{\mathbb{R}^d}\big(f(x)-f_*\big)
  \big(x^{(i)}-x_k^{(i)}\big)\varphi_k(x)\ud x \\
  =&\frac{1}{\sigma_k^2}\int_{x_k^{(i)}}^\infty\int_{\mathbb{R}^{d-1}}
  \big(f(x)-f_*\big)\big(x^{(i)}-x_k^{(i)}\big)
  \varphi_k(x)\ud x^{(-i)}\ud x^{(i)} \\
  &-\frac{1}{\sigma_k^2}\int_{-\infty}^{x_k^{(i)}}\int_{\mathbb{R}^{d-1}}
  \big(f(x)-f_*\big)\big(x_k^{(i)}-x^{(i)}\big)
  \varphi_k(x)\ud x^{(-i)}\ud x^{(i)},
\end{align*}
where $\ud x^{(-i)}=\ud x/\ud x^{(i)}$; together with Assumption \ref{DFD:ass:A}, i.e.,
\begin{equation*}
  \frac{l}{2}\|x-x_*\|_2^2\leqslant
  f(x)-f_*\leqslant\frac{L}{2}\|x-x_*\|_2^2,
\end{equation*}
this yields
\begin{align*}
  \nabla^{(i)}f_{\sigma_k}(x_k)
  \geqslant&\frac{l}{2\sigma_k^2}\int_{x_k^{(i)}}^\infty
  \int_{\mathbb{R}^{d-1}}\|x-x_*\|_2^2\big(x^{(i)}-x_k^{(i)}\big)
  \varphi_k(x)\ud x^{(-i)}\ud x^{(i)} \\
  &-\frac{L}{2\sigma_k^2}\int_{-\infty}^{x_k^{(i)}}
  \int_{\mathbb{R}^{d-1}}\|x-x_*\|_2^2
  \big(x_k^{(i)}-x^{(i)}\big)\varphi_k(x)\ud x^{(-i)}\ud x^{(i)} \\
  =&\frac{l}{2\sigma_k^2}\int_{\mathbb{R}^d}\|x-x_*\|_2^2
  \big(x^{(i)}-x_k^{(i)}\big)\varphi_k(x)\ud x \\
  &-\frac{L-l}{2\sigma_k^2}\int_{-\infty}^{x_k^{(i)}}
  \int_{\mathbb{R}^{d-1}}\|x-x_*\|_2^2
  \big(x_k^{(i)}-x^{(i)}\big)\varphi_k(x)\ud x^{(-i)}\ud x^{(i)}.
\end{align*}
Similarly, since
\begin{equation*}
  \int_{\mathbb{R}^d}\big(x^{(i)}-x_k^{(i)}\big)\varphi_k(x)\ud x=
  \int_{\mathbb{R}^d}\big(x^{(i)}-x_k^{(i)}\big)^3\varphi_k(x)\ud x=0,
\end{equation*}
it holds that
\begin{align*}
  &\frac{l}{2\sigma_k^2}\int_{\mathbb{R}^d}\|x-x_*\|_2^2 \big(x^{(i)}-x_k^{(i)}\big)\varphi_k(x)\ud x \\
  =&\frac{l}{2\sigma_k^2}\int_{\mathbb{R}^d}\big(x^{(i)}-x_*^{(i)}\big)^2
  \big(x^{(i)}-x_k^{(i)}\big)\varphi_k(x)\ud x \\
  =&\frac{l}{2\sigma_k^2}\!\int_\mathbb{R}\!\big[
  \big(x^{(i)}\!\!-\!x_k^{(i)}\big)^3\!\!+\!2
  \big(x_k^{(i)}\!\!-\!x_*^{(i)}\big)
  \big(x^{(i)}\!\!-\!x_k^{(i)}\big)^2\!\!+\!
  \big(x_k^{(i)}\!\!-\!x_*^{(i)}\big)^2\big(x^{(i)}\!\!-\!x_k^{(i)}\big)
  \big]\phi_k(x^{(i)})\ud x^{(i)} \\
  =&l\big(x_k^{(i)}-x_*^{(i)}\big).
\end{align*}
Hence, we obtain
\begin{equation*}
  -\nabla^{(i)}f_{\sigma_k}(x_k)\leqslant
  -l\big(x_k^{(i)}\!-\!x_*^{(i)}\big)+
  \frac{L-l}{2\sigma_k^2}\int_{-\infty}^{x_k^{(i)}}\!
  \int_{\mathbb{R}^{d-1}}\!\|x\!-\!x_*\|_2^2
  \big(x_k^{(i)}\!-\!x^{(i)}\big)\varphi_k(x)\ud x^{(-i)}\ud x^{(i)}.
\end{equation*}
Further, noting that
\begin{align*}
  &\int_\mathbb{R}\big(x^{(j)}-x_*^{(j)}\big)^2
  \phi_k(x^{(j)})\ud x^{(j)} \\
  =&\int_\mathbb{R}\big(x^{(j)}-x_k^{(j)}+x_k^{(j)}-x_*^{(j)}\big)^2
  \phi_k(x^{(j)})\ud x^{(j)} \\
  =&\int_\mathbb{R}\big(x^{(j)}-x_k^{(j)}\big)^2
  \phi_k(x^{(j)})\ud x^{(j)}
  +2\big(x_k^{(j)}-x_*^{(j)}\big)\int_\mathbb{R}
  \big(x^{(j)}-x_k^{(j)}\big)\phi_k(x^{(j)})\ud x^{(j)} \\
  &+\big(x_k^{(j)}-x_*^{(j)}\big)^2\int_\mathbb{R}
  \phi_k(x^{(j)})\ud x^{(j)}\\
  =&\sigma_k^2+\big(x_k^{(j)}-x_*^{(j)}\big)^2,
\end{align*}
it follows that
\begin{align*}
  &\frac{1}{\sigma_k^2}\sum_{j\neq i}\int_{-\infty}^{x_k^{(i)}}
  \int_{\mathbb{R}^{d-1}}\big(x^{(j)}-x_*^{(j)}\big)^2
  \big(x_k^{(i)}-x^{(i)}\big)\varphi_k(x)\ud x^{(-i)}\ud x^{(i)} \\
  =&\frac{1}{\sigma_k^2}\int_{-\infty}^{x_k^{(i)}}
  \big(x_k^{(i)}-x^{(i)}\big)\phi_k(x^{(i)})\ud x^{(i)}
  \cdot\sum_{j\neq i}
  \left[\sigma_k^2+\big(x_k^{(j)}-x_*^{(j)}\big)^2\right] \\
  \leqslant&\frac{1}{\sqrt{2\pi}\sigma_k}
  \left(\sigma_k^2d+\|x_k-x_*\|_2^2\right).
\end{align*}
Hence, it follows from $\sigma_k=\rho^{\frac{k}{2}}\lambda^{-\frac{1}{2}}$ and $\|x_k-x_*\|_2^2\leqslant\rho^kM$ that
\begin{equation*}
  \frac{1}{\sigma_k^2}\sum_{j\neq i}\int_{-\infty}^{x_k^{(i)}}
  \int_{\mathbb{R}^{d-1}}\big(x^{(j)}-x_*^{(j)}\big)^2
  \big(x_k^{(i)}-x^{(i)}\big)\varphi_k(x)\ud x^{(-i)}\ud x^{(i)}
  \leqslant\frac{\rho^{\frac{k}{2}}}{\sqrt{2\pi}}
  \left(\lambda^{-\frac{1}{2}}d+\frac{M}{\lambda^{-\frac{1}{2}}}\right).
\end{equation*}
Similarly, we have
\begin{align*}
  &\frac{1}{\sigma_k^2}\int_{-\infty}^{x_k^{(i)}}
  \int_{\mathbb{R}^{d-1}}\big(x^{(i)}-x_*^{(i)}\big)^2
  \big(x_k^{(i)}-x^{(i)}\big)\varphi_k(x)\ud x^{(-i)}\ud x^{(i)} \\
  =&\frac{1}{\sigma_k^2}\int_{-\infty}^{x_k^{(i)}}
  \big(x^{(i)}-x_*^{(i)}\big)^2\big(x_k^{(i)}-x^{(i)}\big)
  \phi_k(x^{(i)})\ud x^{(i)} \\
  =&\frac{1}{\sigma_k^2}\int_{-\infty}^{x_k^{(i)}}
  \big(x^{(i)}-x_k^{(i)}+x_k^{(i)}-x_*^{(i)}\big)^2
  \big(x_k^{(i)}-x^{(i)}\big)\phi_k(x^{(i)})\ud x^{(i)} \\
  =&\frac{1}{\sigma_k^2}\int_{-\infty}^{x_k^{(i)}}
  \big(x_k^{(i)}-x^{(i)}\big)^3\phi_k(x^{(i)})\ud x^{(i)}
  -2\frac{x_k^{(i)}-x_*^{(i)}}{\sigma_k^2}\int_{-\infty}^{x_k^{(i)}}
  \big(x_k^{(i)}-x^{(i)}\big)^2\phi_k(x^{(i)})\ud x^{(i)} \\
  &+\frac{\big(x_k^{(i)}-x_*^{(i)}\big)^2}{\sigma_k^2}
  \int_{-\infty}^{x_k^{(i)}}
  \big(x_k^{(i)}-x^{(i)}\big)\phi_k(x^{(i)})\ud x^{(i)} \\
  =&\frac{2\sigma_k}{\sqrt{2\pi}}-\big(x_k^{(i)}-x_*^{(i)}\big)
  +\frac{1}{\sqrt{2\pi}\sigma_k}\big(x_k^{(i)}-x_*^{(i)}\big)^2 \\
  \leqslant&-\big(x_k^{(i)}-x_*^{(i)}\big)
  +\frac{\rho^{\frac{k}{2}}}{\sqrt{2\pi}}
  \left(2\lambda^{-\frac{1}{2}}+\frac{M}{\lambda^{-\frac{1}{2}}}\right).
\end{align*}
Finally, we obtain
\begin{align*}
  -\nabla^{(i)}f_{\sigma_k}(x_k)
  \leqslant&-\frac{L+l}{2}\big(x_k^{(i)}-x_*^{(i)}\big)
  +\frac{\rho^{\frac{k}{2}}(L-l)}{\sqrt{2\pi}}
  \left((d+2)\lambda^{-\frac{1}{2}}
  +\frac{2M}{\lambda^{-\frac{1}{2}}}\right) \\
  \leqslant&-\frac{L+l}{2}\big(x_k^{(i)}-x_*^{(i)}\big)
  +\rho^{\frac{k}{2}}(L-l)\frac{\sqrt{(d+2)M}}{\sqrt{\pi}},
\end{align*}
where $\lambda^{-1}=\frac{2M}{d+2}$.

Similarly, we ca prove that
\begin{align*}
  -\nabla^{(i)}f_{\sigma_k}(x_k)
  \geqslant-\frac{L+l}{2}\big(x_k^{(i)}-x_*^{(i)}\big)
  -\rho^{\frac{k}{2}}(L-l)\frac{\sqrt{(d+2)M}}{\sqrt{\pi}},
\end{align*}
and the proof is complete.\qed
\end{proof}

\begin{lemma}\label{DFD:lem:Vkbound}
Under Assumption \ref{DFD:ass:A}, suppose that $\sigma_k=\rho^{\frac{k}{2}}\lambda^{-\frac{1}{2}}$ with $0<\rho<1$ and there is an $M>0$ such that $\|x_k-x_*\|_2^2\leqslant\rho^kM$. Then for any $1\leqslant i\leqslant d$, the variance of $g_k^{(i)}$ satisfies the following inequality:
\begin{equation*}
  \mathbb{V}\big[g_k^{(i)}\big]\leqslant
  \rho^k\frac{L^2}{n}\left(\frac{d+2}{\lambda}+M\right),
\end{equation*}
where $g_k^{(i)}$ is the $i$-th component of $g_k\in\mathbb{R}^d$.
\end{lemma}
\begin{proof}
According to Assumption \ref{DFD:ass:A}, for any $1\leqslant s\leqslant n$, we have
\begin{equation*}
  f(\theta_{k,s})-\min_{1\leqslant j\leqslant n}f(\theta_{k,j})
  \leqslant f(\theta_{k,s})-f_*\leqslant
  \frac{L}{2}\|\theta_{k,s}-x_*\|_2\leqslant
  \frac{L\sigma_k^2}{2}\|\xi_s\|_2+\frac{L}{2}\|x_k-x_*\|_2,
\end{equation*}
then for any $1\leqslant i\leqslant d$, together with the definition of $g_\sigma(x,\xi)$, we further obtain
\begin{align*}
  \mathbb{V}_\xi\big[g_\sigma^{(i)}(x,\xi)\big]
  \leqslant&\mathbb{E}_\xi\big[g_\sigma^{(i)}(x,\xi)\big]^2\\
  =&\mathbb{E}_\xi\bigg[\frac{\big(f(\theta_k)-\min f(\theta_k)\big)
  \xi^{(i)}}{\sigma_k}\bigg]^2 \\
  \leqslant&\mathbb{E}_\xi\bigg[\frac{L\sigma_k}{2}\|\xi_s\|_2\xi^{(i)}
  +\frac{L\|x_k-x_*\|_2}{2\sigma_k}\xi^{(i)}\bigg]^2 \\
  \leqslant&L^2\sigma_k^2\mathbb{E}_\xi
  \big[\|\xi\|_2\xi^{(i)}\big]^2
  +\frac{L^2\|x_k-x_*\|_2^2}{\sigma_k^2}
  \mathbb{E}_\xi\big[\xi^{(i)}\big]^2 \\
  \leqslant&L^2\rho^k\lambda^{-1}\mathbb{E}_\xi
  \big[\|\xi\|_2\xi^{(i)}\big]^2+L^2\rho^kM.
\end{align*}
Since
\begin{equation*}
  \mathbb{E}_\xi\big[\|\xi\|_2\xi^{(i)}\big]^2
  =\sum_{j=1}^d\frac{1}{(\sqrt{2\pi})^d}
  \int_{\mathbb{R}^d}\big(\xi^{(j)}\xi^{(i)}\big)^2
  e^{-\frac{\|\xi\|_2^2}{2}}\ud\xi=d+2,
\end{equation*}
it follows that
\begin{equation*}
  \mathbb{V}\big[g_k^{(i)}\big]\leqslant
  \rho^k\frac{L^2}{n}\left(\frac{d+2}{\lambda}+M\right),
\end{equation*}
and the proof is complete.\qed
\end{proof}

\begin{theorem}\label{DFD:thm:main1}
Under Assumption \ref{DFD:ass:A}, suppose that $\sigma_k=\rho^{\frac{k}{2}}\lambda^{-\frac{1}{2}}$ with $0<\rho<1$ and there is an $M>0$ such that $\|x_s-x_*\|_2^2\leqslant\rho^sM$ for all $1\leqslant s\leqslant k$. If the FD-DFD method (Algorithm \ref{DFD:alg:DFD}) is run with a stepsize parameter $\alpha>0$ such that
\begin{equation*}
  \rho_\alpha=2\left[1-\frac{\alpha(L+l)}{2}\right]^2<\rho<1,
\end{equation*}
then with probability at least $1-\frac{1}{C^2}$, the iterates of FD-DFD satisfy:
\begin{equation*}
  \big\|x_{k+1}-x_*\big\|_2^2\leqslant
  \rho_\alpha^k\big\|x_1\!-\!x_*\big\|_2^2
  +\rho^k\frac{2\rho_\alpha\alpha^2K_C^2}{\rho-\rho_\alpha},
\end{equation*}
where the constant
\begin{equation*}
  K_C=(L-l)\sqrt{\frac{(d+2)M}{\pi}}
  +L\frac{C}{\sqrt{n}}\sqrt{\frac{d+2}{\lambda}+M}
\end{equation*}
is independent of $k$ but depends on $C$.
\end{theorem}
\begin{proof}
According to \eqref{DFD:eq:gki}, Lemmas \ref{DFD:lem:gkbound} and \ref{DFD:lem:Vkbound}, it holds that for any $1\leqslant i\leqslant d$ and all $C>0$, with probability at least $1-\frac{1}{C^2}$,
\begin{equation*}
  \bigg|-g_k^{(i)}+\frac{L+l}{2}\big(x_k^{(i)}-x_*^{(i)}\big)\bigg|
  \leqslant\rho^{\frac{k}{2}}K_C,
\end{equation*}
Together with \eqref{DFD:eq:xki}, we obtain
\begin{align*}
  \Big|x_{k+1}^{(i)}-x_*^{(i)}\Big|
  =&~\Big|x_k^{(i)}-x_*^{(i)}-\alpha g_k^{(i)}\Big| \\
  =&~\Big|x_k^{(i)}-x_*^{(i)}-\frac{\alpha(L+l)}{2}
  \big(x_k^{(i)}-x_*^{(i)}\big)+\Big(-\alpha g_k^{(i)}
  +\frac{\alpha(L+l)}{2}\big(x_k^{(i)}-x_*^{(i)}\big)\Big)\Big| \\
  \leqslant&~\Big|x_k^{(i)}-x_*^{(i)}-\frac{\alpha(L+l)}{2}
  \big(x_k^{(i)}-x_*^{(i)}\big)\Big|+\rho^{\frac{k}{2}}\alpha K_C \\
  =&\left|1-\frac{\alpha(L+l)}{2}\right|\Big|x_k^{(i)}-x_*^{(i)}\Big|
  +\rho^{\frac{k}{2}}\alpha K_C,
\end{align*}
further, by the Arithmetic Mean Geometric Mean inequality, we have
\begin{align*}
  \big(x_{k+1}^{(i)}-x_*^{(i)}\big)^2
  \leqslant&2\left[1-\frac{\alpha(L+l)}{2}\right]^2
  \big(x_k^{(i)}-x_*^{(i)}\big)^2+\rho^k\cdot2\alpha^2K_C^2 \\
  =&\rho_\alpha\big(x_k^{(i)}-x_*^{(i)}\big)^2
  +\rho^k\cdot2\alpha^2K_C^2.
\end{align*}
Finally, by summing $i$ from $1$ to $d$, one obtains
\begin{equation*}
  \big\|x_{k+1}-x_*\big\|_2^2\leqslant\rho_\alpha
  \big\|x_k-x_*\big\|_2^2+\rho^k\cdot2\alpha^2K_C^2,
\end{equation*}
and doing it recursively, one further obtains,
\begin{align*}
  \big\|x_{k+1}\!-\!x_*\big\|_2^2\leqslant&
  \rho_\alpha^k\big\|x_1\!-\!x_*\big\|_2^2
  +\rho^k\left(1\!+\!\frac{\rho_\alpha}{\rho}\!+\!\cdots\!+\!
  \frac{\rho_\alpha^{k-1}}{\rho^{k-1}}\right)2\alpha^2K_C^2 \\
  \leqslant&\rho_\alpha^k\big\|x_1\!-\!x_*\big\|_2^2
  +\rho^k\frac{2\rho_\alpha\alpha^2K_C^2}{\rho-\rho_\alpha},
\end{align*}
and the proof is complete.\qed
\end{proof}

The following theorem states that when the parameters are properly selected,  the iterates of FD-DFD satisfy $\|x_{k+1}-x_*\|_2^2\leqslant\rho^k\big\|x_1-x_*\big\|_2^2$ for all $k\in\mathbb{N}$ in probability.
\begin{theorem}\label{DFD:thm:main2}
Under Assumption \ref{DFD:ass:A}, suppose that $\sigma_k=\rho^{\frac{k}{2}}\lambda^{-\frac{1}{2}}$ with $0<\rho<1$. If the FD-DFD method (Algorithm \ref{DFD:alg:DFD}) is run with a stepsize parameter $\alpha>0$ such that
\begin{equation*}
  \frac{\rho_\alpha}{\rho}\left[1+
  \frac{2\alpha^2K_C^2}{\rho-\rho_\alpha}\right]<1,
\end{equation*}
then with high probability, the iterates of FD-DFD satisfy for all $k\in\mathbb{N}$:
\begin{equation*}
  \|x_{k+1}-x_*\|_2^2\leqslant\rho^k\big\|x_1-x_*\big\|_2^2,
\end{equation*}
where $\rho_\alpha=2\left[1-\frac{\alpha(L-l)}{2}\right]^2$ and
\begin{equation*}
  K_C=(L-l)\sqrt{\frac{(d+2)\|x_1-x_*\|_2^2}{\pi\rho}}
  +L\frac{C}{\sqrt{n}}\sqrt{\frac{d+2}{\lambda}
  +\frac{\|x_1-x_*\|_2^2}{\rho}}.
\end{equation*}
\end{theorem}
\begin{proof}
Let $M=\rho^{-1}\|x_1-x_*\|_2^2$, that is, $\|x_1-x_*\|_2^2\leqslant\rho M$, which satisfies the condition of Theorem \ref{DFD:thm:main1}. Using Theorem \ref{DFD:thm:main1} and induction, one can deduce
\begin{equation*}
  \big\|x_{k+1}-x_*\big\|_2^2\leqslant
  \rho^{k+1}M\left(\frac{\rho_\alpha^k}{\rho^k}
  +\frac{2\rho_\alpha\alpha^2K_C^2}{\rho(\rho-\rho_\alpha)}\right).
\end{equation*}
Notice that for every $k\in\mathbb{N}$, it follows that
\begin{equation*}
  \frac{\rho_\alpha^k}{\rho^k}
  +\frac{2\rho_\alpha\alpha^2K_C^2}{\rho(\rho-\rho_\alpha)}
  <\frac{\rho_\alpha}{\rho}\left[1+
  \frac{2\alpha^2K_C^2}{\rho-\rho_\alpha}\right]<1,
\end{equation*}
thus, one can finally obtain
\begin{equation*}
  \big\|x_{k+1}-x_*\big\|_2^2\leqslant
  \rho^{k+1}M=\rho^k\big\|x_1-x_*\big\|_2^2,
\end{equation*}
and the proof is complete.\qed
\end{proof}

Since $n$ is independent of $k$, the following total work complexity bound for the FD-DFD method is immediate from Theorem \ref{DFD:thm:main2}.
\begin{corollary}[Complexity bound]
Suppose the conditions of Theorem \ref{DFD:thm:main2} hold. Then the
number of function evaluations of the FD-DFD (Algorithm \ref{DFD:alg:DFD}) required to achieve $\|x_k-x_*\|_2^2\leqslant\epsilon$ is $\mathcal{O}(\log(1/\epsilon))$.
\end{corollary}

\section{Comparison of RAD and FD-DFD}
\label{DFD:s3}

With an initial point $x_1$, three fixed parameters $\lambda>0$, $0<\rho<1$ and $n\in\mathbb{N}$, the RAD method \cite{LuoX2020A_RAD} is characterized by the iteration
\begin{equation}\label{DFD:eq:RAD}
  x_{k+1}=\frac{\sum_{i=1}^n\theta_{k,i} \exp[-m_k^{-1}(f(\theta_{k,i})-f_*)]}{\sum_{i=1}^n \exp[-m_k^{-1}(f(\theta_{k,i})-f_*)]}.
\end{equation}
With an additional stepsize parameter $\alpha>0$, the FD-DFD method is characterized by the iteration
\begin{equation}\label{DFD:eq:DFD}
  x_{k+1}=x_k-\frac{\alpha}{nm_k}\sum_{i=1}^n
  \bigg(f(\theta_{k,i})-\min_{1\leqslant j\leqslant n}
  f(\theta_{k,i})\bigg)(\theta_{k,i}-x_k),
\end{equation}
where $\theta_{k,i}\sim\mathcal{N}(x_k,\rho^{k}\lambda^{-1}I_d)$ and $m_k^2=\mathbb{E}[(f(\theta_{k,i})-f_*)^2]$. In practice, $f_*$ and $m_k$ should be replaced with corresponding estimates.

Both the RAD and FD-DFD method have their own characteristics. Since RAD is based on the asymptotic representation for the solution of regularized minimization, it does not require a stepsize parameter, or in other words, it can automatically obtain the optimal stepsize; however, the parameter $n$ in RAD algorithms increases as the dimension $d$ increases (See Fig. 3 in \cite{LuoX2020A_RAD}). In comparison, the parameter $n$ in FD-DFD algorithms is almost independent of $d$ because the variance term containing $n$ is insignificance for an appropriate $n$ (See Fig. \ref{DFD:fig:E2} and Theorem \ref{DFD:thm:main1}); but as a price, there is an additional stepsize parameter $\alpha$ and the choice of $\alpha$ directly affects whether the iterate sequence will converge to the global minimizer.

\section{Numerical experiments}
\label{DFD:s4}

We illustrate the performance of the FD-DFD algorithm by considering the revised Rastrigin function in $\mathbb{R}^d$ defined as
\begin{equation*}
  f(x)=\|x\|_2^2-\frac{1}{2}\sum_{i=1}^d\cos\big(5\pi x^{(i)}\big)
  +\frac{d}{2},~~\textrm{where $x^{(i)}$ be the $i$-th component of $x$}.
\end{equation*}
As shown in Fig. \ref{DFD:fig:A}, this revised Rastrigin function satisfies Assumption \ref{DFD:ass:A}. It has a unique global minima located at the origin and many local minima, e.g., the number of its local minima reaches $5^d$ in the hypercube $[-1,1]^d$.

Fig. \ref{DFD:fig:E2} shows the performance of FD-DFD algorithms for the revised Rastrigin function in various dimensions from $5$ to $500$. These experiments clearly demonstrate the global linear convergence of the FD-DFD method.

\begin{figure}[tbhp]
\centering
\subfigure{\includegraphics[width=0.325\textwidth]{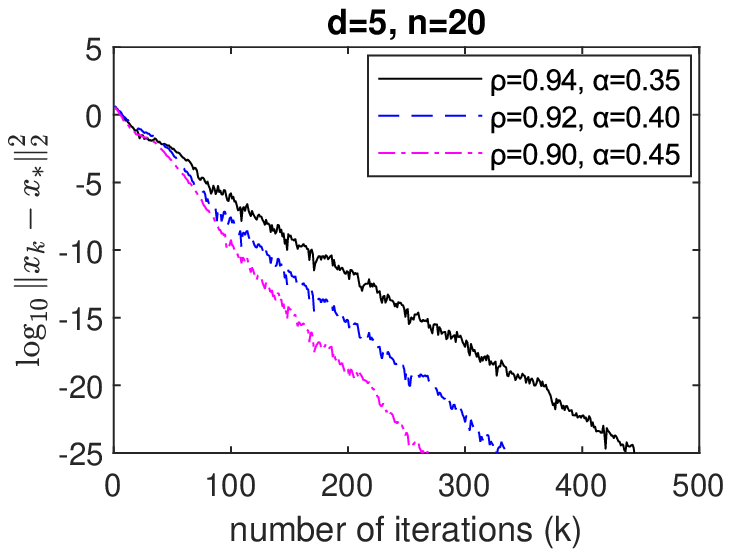}}
\subfigure{\includegraphics[width=0.325\textwidth]{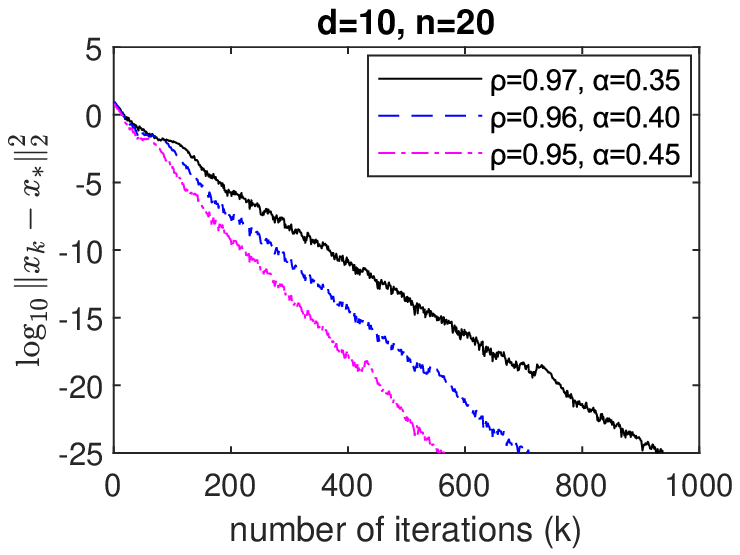}}
\subfigure{\includegraphics[width=0.325\textwidth]{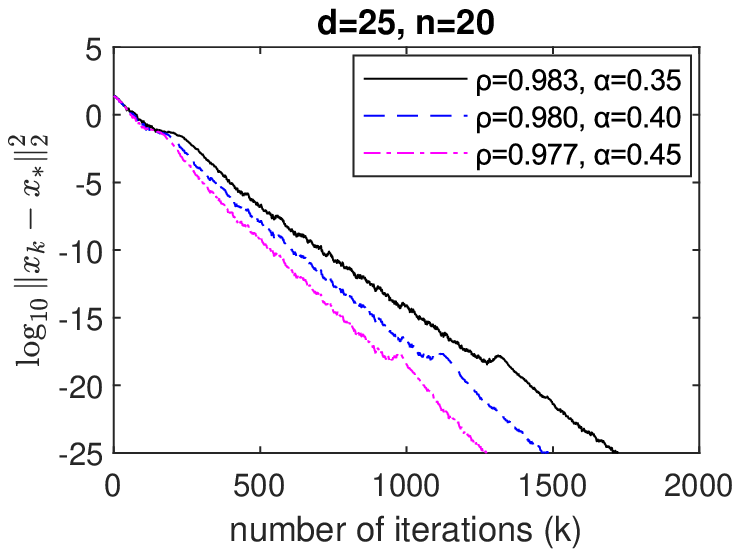}}
\subfigure{\includegraphics[width=0.325\textwidth]{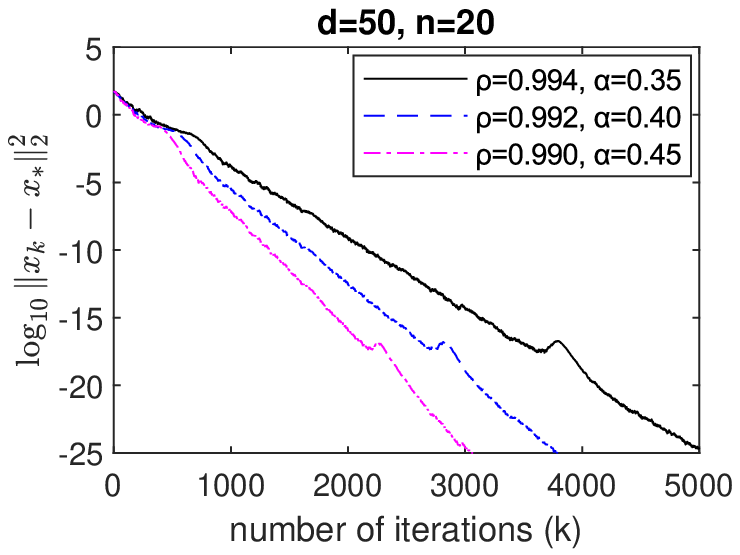}}
\subfigure{\includegraphics[width=0.325\textwidth]{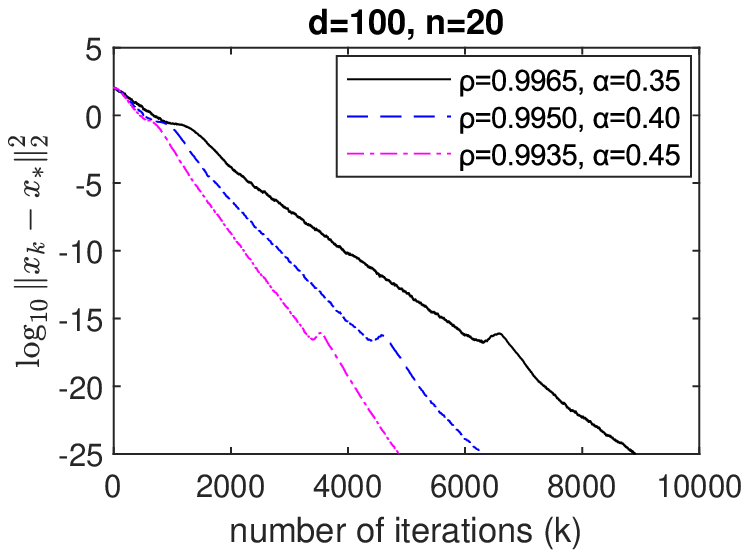}}
\subfigure{\includegraphics[width=0.325\textwidth]{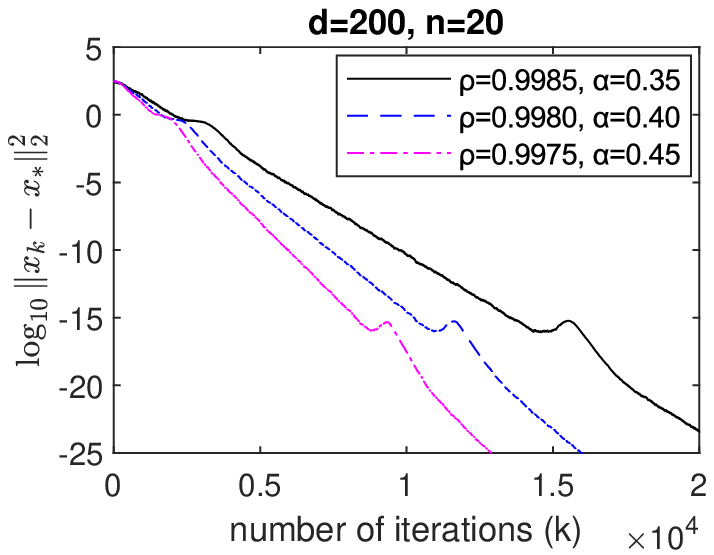}}
\subfigure{\includegraphics[width=0.325\textwidth]{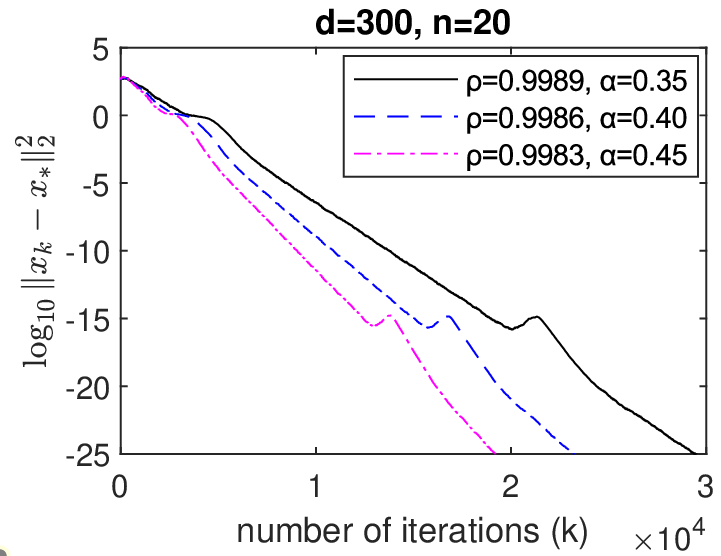}}
\subfigure{\includegraphics[width=0.325\textwidth]{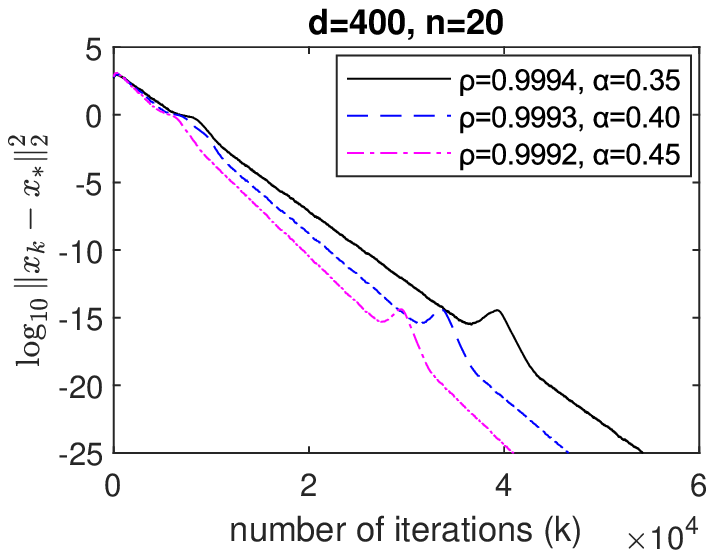}}
\subfigure{\includegraphics[width=0.325\textwidth]{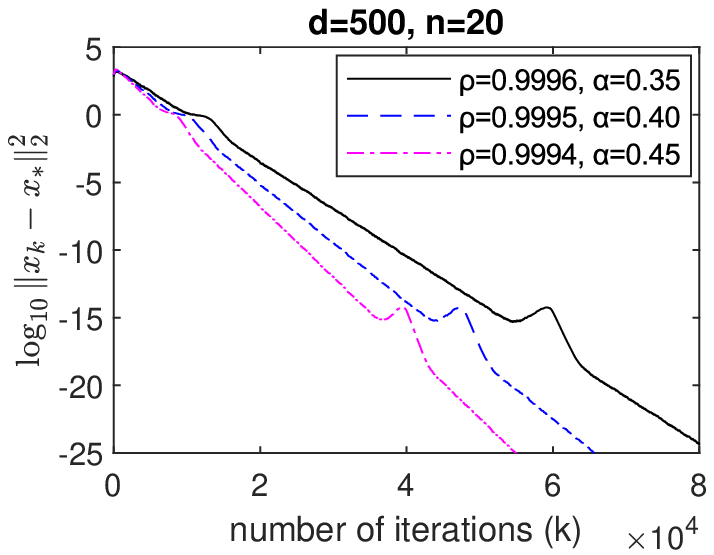}}
\caption{Performance of the FD-DFD method for the revised Rastrigin function in various dimensions, every initial iterate is randomly selected on a sphere of radius $\sqrt{d}$ centered at the origin, the parameter $\lambda=1/\sqrt{d}$, three different settings for the parameters $\rho$ and $\alpha$ are run independently for each plot.}
\label{DFD:fig:E2}
\end{figure}

In these experiments, every initial iterate is randomly selected on a $d$-dimensional sphere of radius $\sqrt{d}$ centered at the origin. Furthermore, the random vectors in each iteration are sequentially generated by a halton sequence with RR$2$ scramble type \cite{KocisL1997A_QMCscramble}. The algorithm is implemented in Matlab. The source code of the implementation is available at https://github.com/xiaopengluo/dfd.

\section{Conclusions}
\label{DFD:s5}

In this work we have analyzed that the finite-difference derivative-free descent (FD-DFD) method enjoys linear convergence for finding the global minima of a class of multiple minima functions. It also has a total work complexity bound $\mathcal{O}(\log\frac{1}{\epsilon})$ to find a point such that the gap between this point and the global minimizer is less than $\epsilon$. Numerical experiments in various dimensions demonstrate all the benefits.

\begin{acknowledgements}
We thank Prof. Herschel A. Rabitz for several discussions about global optimization for multiple minima problems. 
\end{acknowledgements}



\bibliographystyle{spmpsci}
\bibliography{MReferences}


\end{document}